\pdfoutput=1
\documentclass[11pt]{amsart}

\usepackage{hyperref} 

\hypersetup{
  pdftitle={Tetrachromagea},
  pdfauthor={Jimmy Dillies},
  pdfsubject={},
  pdfkeywords={},
  colorlinks=true,
  breaklinks=true,
  bookmarksopen=true,
  bookmarksnumbered=true,
  pdfpagemode=UseOutlines,
  plainpages=false,
  linkcolor=blue,
  anchorcolor=green}

\newtheorem{thm}{Theorem}[section]
\newtheorem{prop}{Proposition}[section]
\newtheorem{lem}{Lemma}[section]
\newtheorem{pro}{Properties}[section]
\newtheorem{defi}{Definition}[section]
\newtheorem{rmk}{Remark}[section]
\newtheorem{ex}{Example}[section]

\usepackage{amsmath}

\usepackage{graphicx}

\newcommand{\IF}{\mathbb{F}}

\newcommand{\fd}{\mathfrak{d}}

\title{Tetrachromagea} 

\author{Jimmy Dillies}
\thanks{The author would like to thank Gautier Berck for sharing his thesis which inspired this work~\cite{Ber}.}
\dedicatory{Dedicated to J.Dalton.}
\date{}

\begin{document}


\begin{abstract}
\noindent We construct a moduli space of four colorings on planar cubic graphs. 
More precisely, we introduce the notion of weak Hamiltonian, a generalization of Hamiltonian cycles, and relate it to 4-colorings. 
Weak Hamiltonians have a form of deformation, which we call mutation, which gives them a graph structure, the Weak Hamiltonian graph ($\fd$).
This graph encodes the different colorings as 3 vertex cliques.
Identifying vertices on these cliques, we obtain a new graph, the chromatic graph ($\chi$), whose vertices are exactly the colorings of the original graph.
Also, this construction gives a heuristic argument on why $4$ colors are sufficient to color planar maps. 
\end{abstract}

\maketitle 

\setcounter{tocdepth}{1}
\tableofcontents


\section{Introduction}

The four color problem, or Guthrie's problem, asks whether any planar map can be colored using four shades so that no neighboring regions share the same color.
After a rich, long and notorious history\footnote{See~\cite{BLW} for a detailed account.}, the problem was finally solved by Appel and Haken in 1976~\cite{AH1,AH2}.
They reduced the problem to $1936$ cases which they checked by computer . 
In 1996, Robertson, Sanders, Seymour and Thomas presented a new proof where the computer only had to go through $633$ cases~\cite{RSST}.
In this work, we do not worry about the existence of a coloring but rather about how different colorings are related.
To this purpose, we introduce the concept of weak Hamiltonian. 
Weak Hamiltonians are a loosened form of Hamiltonians in planar cubic graphs and we show how they are actually equivalent to four colorings.
Weak Hamiltonian also have extra structure in form of mutations: each weak Hamiltonian can be transformed in two or more different weak Hamiltonians. 
Mutations endow the set of weak Hamiltonians of a graph structure (the weak Hamiltonian graph, $\fd$) and, hereby, the set of colorings of a graph structure, the chromatic graph $\chi$.
These graphs are the natural framework to discuss problems such as which graphs admit a unique four coloring (see~\cite{Fow} for related discussions).

\subsection{Conventions}

As in Lando and Zvonkin~\cite{LZ}, we define a {\bf map} as a graph embedded on an oriented surface. 
In particular, all edges incident to a same vertex admit a natural cyclic order. 
Moreover, we request that our maps be bridgeless.
Otherwise, we mainly use the language of Diestel~\cite{Die}.

\section{Weak Hamiltonians}

In this section, we simply require our maps to be cubic (i.e. all vertices have degree $3$).

\begin{defi}
A {\bf weak Hamiltonian} is a 2-factor all of whose cycles have even length.
\end{defi}

We will call {\bf weakly Hamiltonian} a graph which admits a weak Hamilton cycle. 
In particular, all Hamiltonian cubic graphs are weakly Hamiltonian as they have an even number of vertices (the number of half edges).

The converse is not true: not all weakly Hamiltonian graphs are Hamiltonian. 
Note that there are graphs, such as the biclique $K_{2,3}$ -- or, for that matter, any graph with an odd number of vertices --, which are neither Hamiltonian, nor weakly Hamiltonian.

\subsection{Mutations}

Let $H$ be a weak Hamilton cycle on $G$ and $\{C_{i}\}_{i=1,\ldots,N}$ its cycles. 
For each cycle $C_{i}$ pick a 1-factor $F_{i}$. 
The complement of a weak Hamiltonian $C$ is a 1-factor of the original map which we denote by $C^{\perp}$.

\begin{defi}
The $(F_{1},\ldots,F_{N})$ {\bf mutation} of $C$ is the map $\mu_{F_{1},\ldots,F_{N}}(C)$ obtained by taking the union of the $F_{i}$ and $C^{\perp}$.
\end{defi}

\begin{ex}
Consider the prism graph $\Pi_4$ with the weak Hamiltonian consisting of the internal and external cycles -- see Figure~\ref{fig:wham}. 
We pick as 1-factors on each cycle the dashed edges.
The mutated weak Hamiltonian is the one on the right of the figure.
\begin{figure}[!h]
{\includegraphics[height=1.5cm]{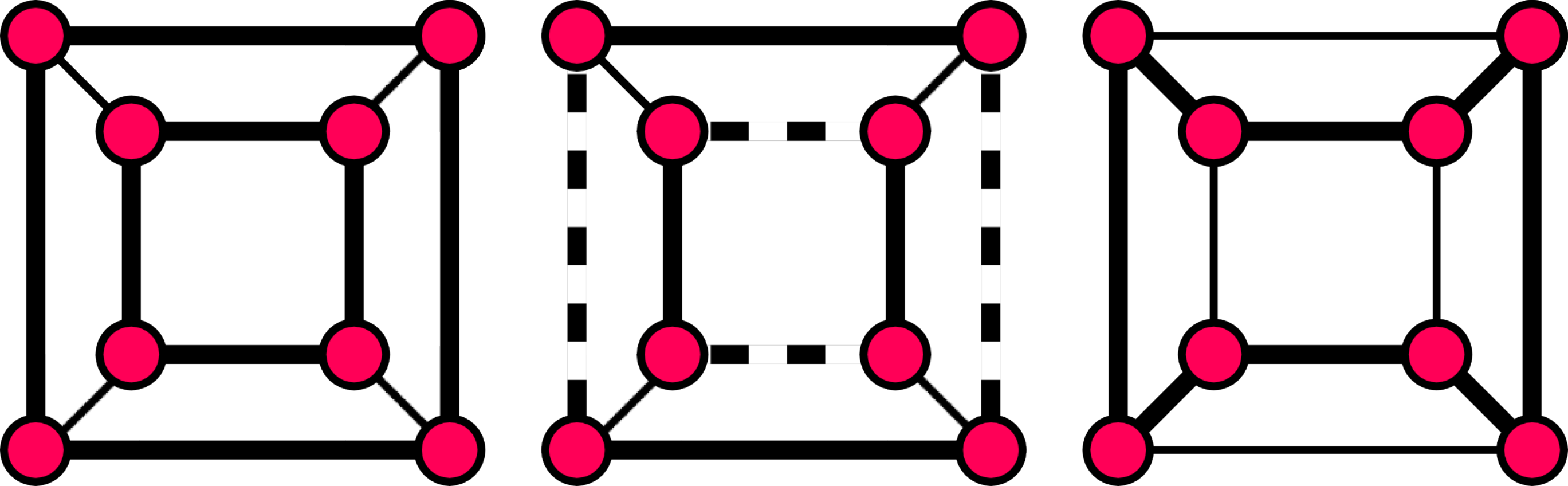}}
\caption{The original weak Hamiltonian is depicted on the left; the 1-factors $F_{1}$, $F_{2}$ are represented as dashed edges in the center and, the mutated weak Hamiltonian is on the right.}
\label{fig:wham}
\end{figure}

\end{ex}

\begin{lem}
\label{lem:freecycle}
Any mutation of a weak Hamiltonian is a weak Hamiltonian. 
\end{lem}

\begin{proof}
It is clear that the mutation is a 2-factor as the criterion is local. 
Moreover all cycles are even since the mutation consists of alternating edges coming from $C^{\perp}$ and the $F_{i}$.
\end{proof}

It follows directly from the definition that 

\begin{pro} Given any cubic map:
\label{pro:properties}
\begin{enumerate} 
\item A map is the union of a weak Hamiltonian and one of its mutations.
\item Each weak Hamiltonian with $N$ cycles admits $2^{N}$ mutations.
\item A weak Hamiltonian $H$ can mutate to $H'$ iff $H'$ can mutate to $H$.
\end{enumerate}
\end{pro}

Note that this is not the first attempt to study the four color problem using modifications of the graph.
However, in general it is the original graph that is deformed and not underlying structures.
In particular, one often tries to reduce the graph by erasing and/or collapsing edges in a manner compatible with the computation of the chromatic polynomial.
In the above construction, the original graph remains unchanged, it is the weak Hamiltonians which are transformed.

\section{Four Coloring}

The first part of this section is well known and shows why it is sufficient to consider cubic graphs when considering 4 colorability.
We include it as the proof has a geometric flavor and, also, as the usual argument refers to the dual set up where colors are assigned to vertices.
In the second part we show how the existence of a four coloring is equivalent to the existence of a weak Hamiltonian.

\subsection{To cubic graphs}

Recall that  a {map} is a graph embedded on a oriented surface. 
We think of the faces as \emph{regions} and of the edges as \emph{boundaries}.
An {\bf $n$-coloring} of our map is function $\phi$ from the set of faces of the graph to the set $J_{n}=\{1,\ldots,n\}$ such that no two neighboring regions have the same image.\footnote{Our definition is slightly archaic and dual to the modern general convention which is to call a coloring a function on the vertices.}

\begin{defi}
The {\bf blowup} of the map $G$ at a vertex $x$ is the map $\tilde{G}_x$ where $x$ is replaced by one vertex per incident edge and a collection of edges joining the new vertices in a cycle oriented as the original edges.
\end{defi}

\begin{figure}[!ht]
{\includegraphics[height=1.5cm]{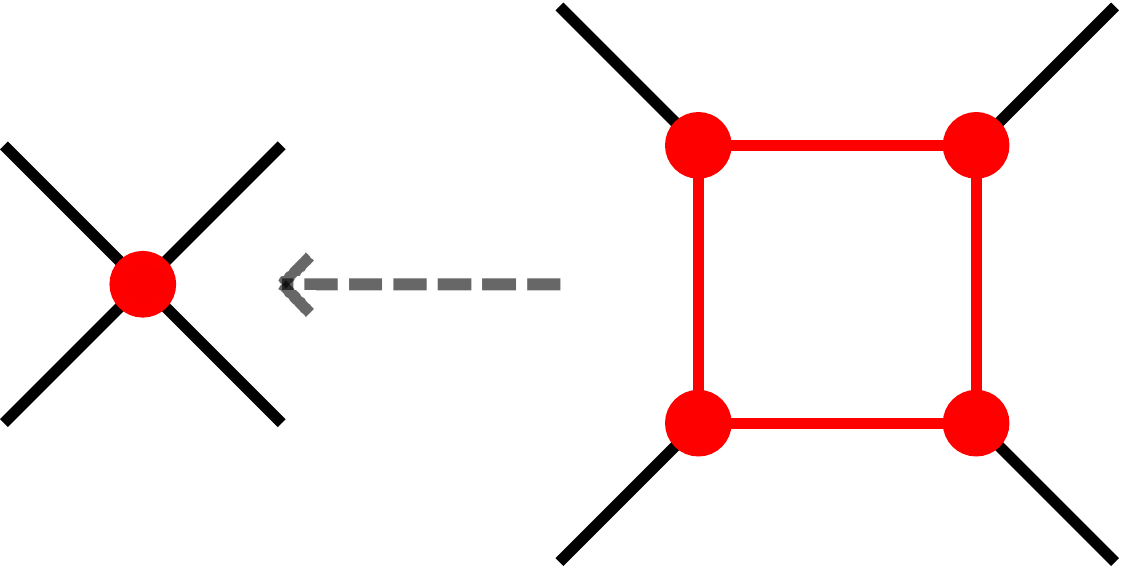}}
\caption{Local picture: the vertex $x$ and its blowup.}
\label{fig:blowup}
\end{figure}

\begin{defi}
Let $G$ be a map, its {\bf resolution} $\tilde{G}$ is the successive blowup at all points of degree $4$ or higher. 
\end{defi}

Note that the resolution is well defined.
The blowup at a point reduces by one the number of vertices of degree at least $4$ and we are thus certain that the process terminates.
Moreover, since the blowup is a local construction, the order of the blowups does not matter. 
The following statement is classic, we have simply rephrased the proofs in terms of blow-ups:

\begin{prop}
All planar maps are four colorable if and only if all maps of degree $3$ are four colorable.
\end{prop}

\begin{proof}
One direction is immediate.
Assume now that all maps of degree $3$ are four colorable.
Let $G$ be a general map and $\tilde{G}$ its resolution.
By hypothesis $\tilde{G}$ admits some $4$-colouring $\phi$.
For a face $F$ of $G$, we write $\tilde{F}$ for the corresponding face on the resolution.
Define the function $\psi$ on the faces of $G$ as 
$$\psi(F):=\phi(\tilde{F})$$
The map $\psi$ is a four coloring of $G$ as two faces  $F_1$ and $F_2$ are neighbors if and only if $\tilde{F_1}$ and $\tilde{F_2}$ are.
\end{proof}

\subsection{Equivalence}

Let us begin by recalling the the Jordan Curve Theorem.

\begin{thm}[Jordan Curve Theorem]
The complement of a simple closed planar curve is made of two components. 
\end{thm}

A simple closed planar curve is often called a {\bf Jordan curve}.
A weak Hamiltonian on a planar graph consists of a disjoint union of Jordan curves. 
While there are now \`a priori more than $2$ regions, we have a milder version of the above theorem:

\begin{prop}
\label{prop:2color}
A weak Hamiltonian on a planar map $G$ is $2$-coloured map.
\end{prop}

\begin{proof}
Let $H=\{C_{1},\ldots, C_{n}\}$ be a weak Hamiltonian.
Each cycle $C_i$ defines an inside and an outside and hence a characteristic function $\chi_{C_i}$ assigning the value $1$ to points of its interior and $0$ on its exterior.
The function $\phi$ defined by
$$\phi(x) \equiv \sum_i \chi_{C_i} (x) \pmod 2$$
gives the requested $2$ coloring.
\end{proof}

\begin{thm}
\label{thm:main}
A planar cubic map is $4$ colorable if and only if it admits a weak Hamiltonian.
\end{thm}

\begin{proof}
Before we start, let us pick a bijection $\gamma : J_{4} \rightarrow J_{2}\times J_{2} $.\\
Assume the graph admits a weak Hamiltonian $H_{1}$.
By Lemma~\ref{lem:freecycle}, there is another weak Hamiltonian $H_{2}$ such that all edges belong to either one of these two cycles.
By the proof of Proposition~\ref{prop:2color}, each cycle defines a $2$-coloring, say $\phi_{1}$ and $\phi_{2}$, on the restricted graphs.
Define a $4$ coloring on the complete map as 
$$\Phi= \gamma^{(-1)} (\phi_{1} \times \phi_{2})$$
Since each edge belongs either to $H_{1}$ or $H_{2}$, we are guaranteed that distinct faces of the map are colored differently through $\Phi$.

Conversely, assume the map is four colored by the function $\Phi$.
Define new functions $\phi_{1,2}$ with image in $J_{2}$ as  
$$\phi_{i}=\pi_{i}\circ \gamma$$
where $\pi_{i}$ represents the projection of $J_{2}\times J_{2} $ onto its i-th component.
We can now define $H_{i}$ as union of the edges between regions of different color with respect to $\phi_{i}$. \\
To see that $H_{i}$ is a 2-factor, note that each vertex of the map has degree $3$ hence each $\phi_{i}$, restricted to the three regions incident to a chosen vertex must be surjective.
Therefore, one of the the three regions has a different image under $\phi_{i}$ than the other $2$ regions. 
The edges of $H_{i}$ incident to the vertex are exactly the two edges bounding this region.\\
Pick a cycle in $H_{i}$.
Given the orientation, we can talk about the inside bounded regions and the outside bounded regions.
As the outside pointing edges belong to $H_{j}$, the regions bounded outside of $H_{i}$ form a sequence of faces whose image alternates with respect to $\phi_{j}$: they come in even number.
The same holds for regions on the inside and we have thus that each cycle of $H_{i}$ has even length.
\end{proof}

\begin{figure}[!h]
{\includegraphics[height=1.5cm]{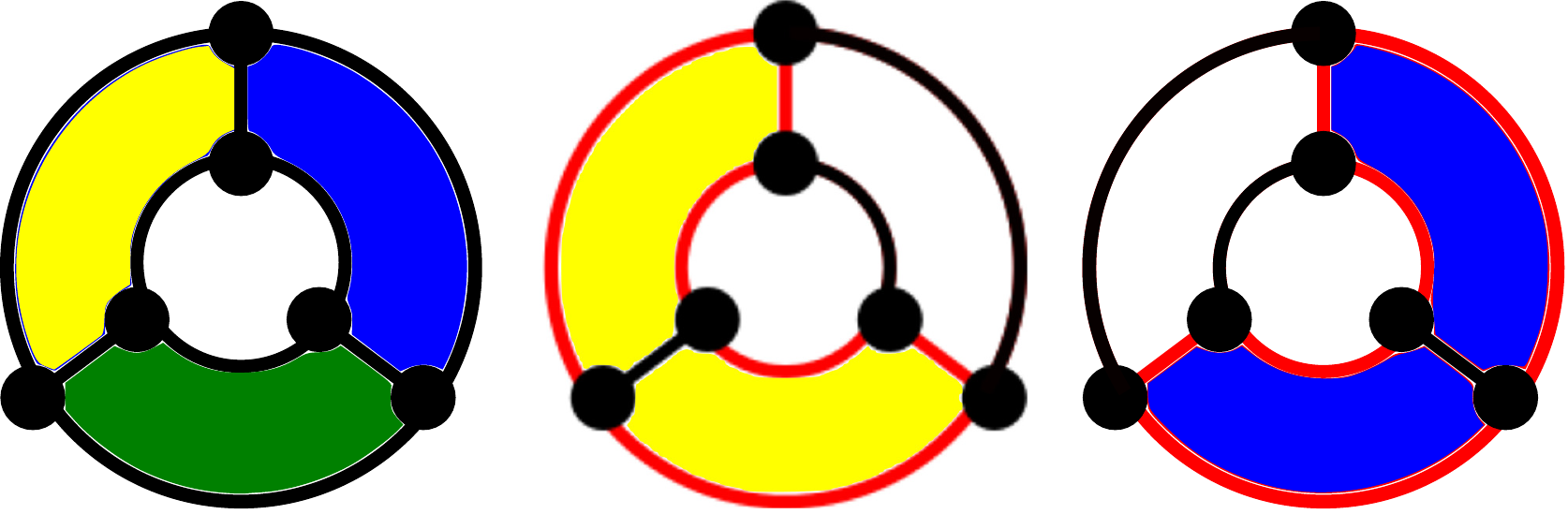}}
\caption{An illustration of Theorem~\ref{thm:main}}
\label{fig:4color}
\end{figure}

\begin{rmk}
\label{rmk:freedom}
In the proof of Theorem~\ref{thm:main}, there is some arbitrariness in the choice of the bijection $\gamma : J_{4} \rightarrow J_{2}\times J_{2} $.
Actually, up to isomorphism, we can see that any coloring gives rise to 3 distinct weak Hamiltonians, all of which are mutation of one another.
\end{rmk}

\section{Moduli}

In this section we use the notion of mutation to endow weak Hamiltonians and, \emph{en passant}, four colorings with a graph structure. 
The properties of these graphs characterize the different colorings that a map can have.

\begin{defi}
Let $G$ be a planar map.
Its {\bf weak Hamiltonian graph} $\fd(G)$ has as vertices the weak Hamiltonians of $G$ and, two vertices are connected by an edge when the associated weak Hamiltonians are related by a mutation.
\end{defi}

\begin{pro}Structure of $\fd(G)$:
\begin{enumerate}
\item Given two neighboring vertices $h_1$, $h_2$ there exists a third vertex $h_3$ neighbor to the original ones such that $h_1 + h_2 + h_3 = 0$ in $\IF_2^{|E|}$.
\item All vertices of $\fd(G)$ come in 3-vertex cliques which we call {\bf chromatic cliques}.
\item The degree of a vertex is $2^N$ where $N$ is the number of its cycles.
\end{enumerate}
\end{pro}

\begin{proof}
The first point follows from Remark~\ref{rmk:freedom}. 
The cliques of the second point are formed by the $h_{i}$ of the first point.
Three is immediate
\end{proof}

We can now associate a second graph to $G$.

\begin{defi}
The {\bf chromatic graph} of $G$, $\chi(G)$ is the graph obtained by contracting the chromatic cliques in $\fd(G)$.
\end{defi}

\begin{ex}
\begin{enumerate}

\item The weak Hamiltonian graph of the 4-prism is given in Figure~\ref{fig:fdgraph}. We see in the weak Hamiltonian graph the $\mathfrak{S}_3$ symmetry of the original cube.

\begin{figure}[!h]
{\includegraphics[height=4in]{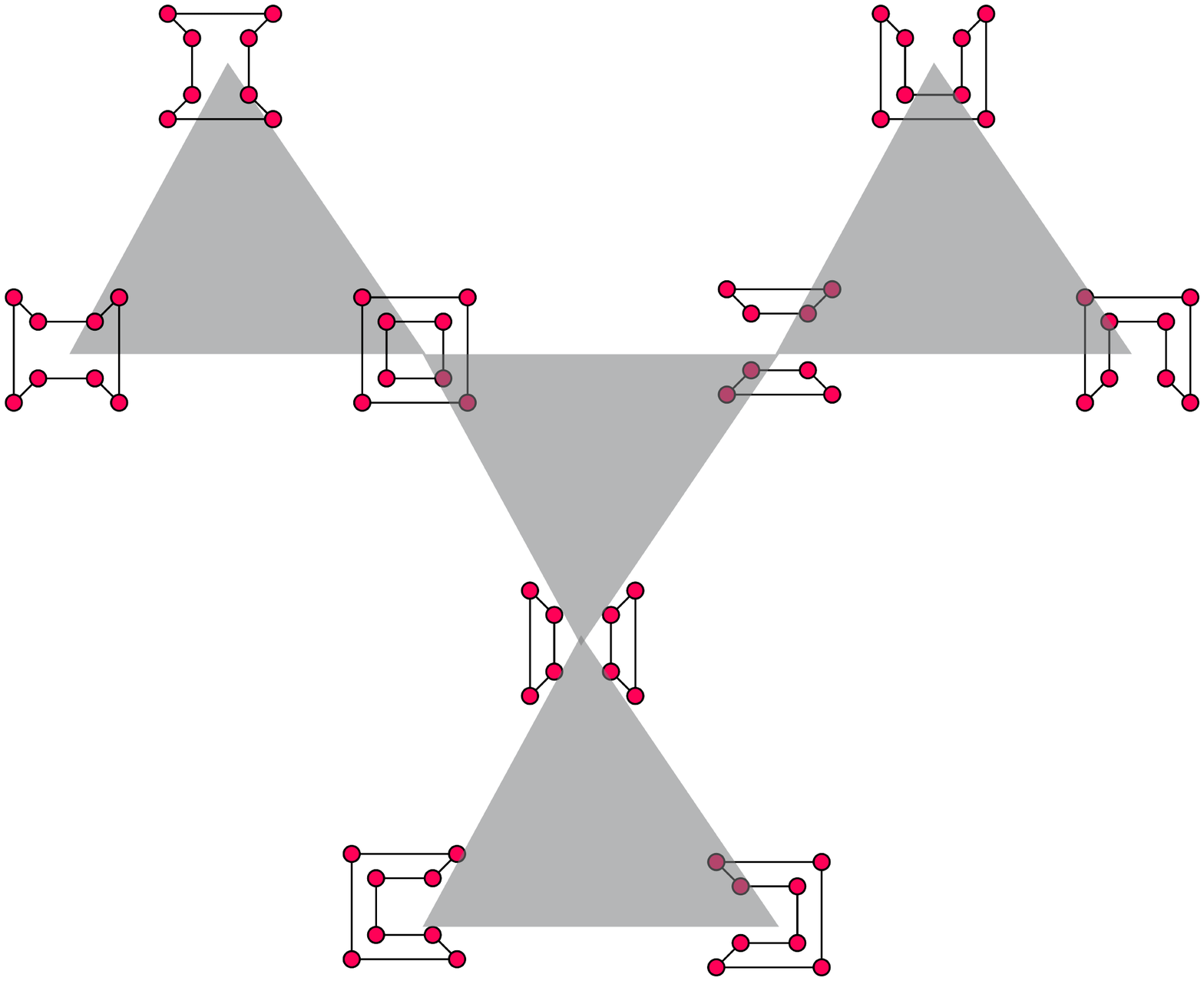}}
\caption{Weak Hamiltonian Graph $\fd\left(\Pi_4\right)$ of the 4-prism. Chromatic cliques are marked in grey.}
\label{fig:fdgraph}
\end{figure}

\item The chromatic graph of the 5-prism is the cyclic graph with five elements : $$\chi\left(\Pi_5\right)=C_5.$$

\end{enumerate}
\end{ex}

\begin{rmk}
A priori, when passing from $\fd$ to $\chi$ there is some information lost beyond the actual weak Hamiltonians.
The chromatic graph ignores how distinct colorings are related through mutations.
For example, there are two possible weak Hamiltonian graphs who contract to a chromatic graph $C_3$: three 3-vertex cliques in a cycle and 3 coincident 3-vertex cliques.
\end{rmk}



\begin{thebibliography}{9}

\bibitem{AH1} K. Appel \& W. Haken.
Every Planar Map is Four Colorable Part I. Discharging.
Illinois Journal of Mathematics 21: 429Ð490. 1977

\bibitem{AH2} K. Appel, W. Haken \& J. Koch.
Every Planar Map is Four Colorable Part II. Reducibility.
Illinois Journal of Mathematics 21: 491Ð567. 1977

\bibitem{Ber}G. Berck. 
Invariants de Vassiliev et repr\'esentation int\'egrale. 
Senior thesis UCL. Louvain-la-Neuve. 1998

\bibitem{BLW}N. Biggs, ; E. Lloyd \& R. Wilson. 
Graph Theory
Oxford. Oxford University Press. 1986

\bibitem{Dal} J. Dalton. 
Extraordinary facts relating to the vision of colours: with observations. 
Memoirs of the Literary and Philosophical Society of Manchester 5: 28-45. 1798

\bibitem{Die} R. Diestel.
Graph Theory.
Electronic Edition 2010. Retrieved from \url{http://diestel-graph-theory.com/}

\bibitem{Fow}T. Fowler.
Unique coloring of planar graphs. 
Ph.D. Thesis Georgia Institute of Technology. Atlanta. 1998

\bibitem{LZ} S. Lando \& A. Zvonkin
Graphs on Surfaces and Their Applications.
Springer-Verlag, 2004

\bibitem{RSST} N. Robertson, Neil, D. Sanders, P. Seymour \& R.Thomas.
The Four-Colour Theorem 
J. Combin. Theory Ser. B 70 (1): 2Ð44. 1997


\end{thebibliography}
\end{document}